\documentclass[12pt,reqno,twoside]{amsart}%

\usepackage{amsmath,amsthm,amscd,amsthm,upref,indentfirst}
\usepackage{amsfonts,mathrsfs}

\usepackage{amssymb,amsbsy,bm}
\usepackage{graphicx}
\usepackage{times}
\usepackage{amsaddr}
\usepackage{soul}

\usepackage[us,long,24hr]{datetime}

\usepackage{mathabx}
\usepackage[usenames,dvipsnames]{color}

\theoremstyle{plain}
\newtheorem{theorem}{Theorem}

\newtheorem{lemma}[theorem]{Lemma}

\theoremstyle{definition}
\newtheorem{definition}[theorem]{Definition}

\theoremstyle{remark}
\newtheorem{remark}[theorem]{Remark}

\newtheorem{example}[theorem]{Example}
\numberwithin{equation}{section}
\numberwithin{theorem}{section}

\usepackage{exscale}
\usepackage[scaled=0.86]{helvet}
\renewcommand{\mathfrak}[1]{{\textbf{\upshape #1}}}
\renewcommand{\mathbf}{\bm}
\renewcommand{\emph}[1]{\textrm{{\upshape #1}}}
\usepackage[scr=euler,scrscaled=1.15,bb=txof,bbscaled=1.16,cal=boondox,calscaled=1.15]{mathalfa}
\renewcommand{\mathit}[1]{\mathscr #1}
\renewcommand{\mathtt}[1]{\scalebox{1}{\bfseries \texttt{\upshape #1}}}

\usepackage{float}

\usepackage[justification=centering]{caption}

\numberwithin{equation}{section}
\numberwithin{theorem}{section}

\usepackage[normalem]{ulem}
\usepackage[square,comma]{natbib}

\usepackage{mparhack}
\usepackage{url}

\usepackage[verbose]{backref}
\backrefsetup{verbose=false}
\renewcommand*{\backref}[1]{}
\renewcommand*{\backrefalt}[4]{[{\tiny%
    \ifcase #1 \textsl{Not cited}%
          \or \textsl{Cited on page}~\textcolor{BrickRed}{#2}%
          \else \textsl{Cited on pages}~\textcolor{BrickRed}{#2}%
    \fi%
    }]}

\usepackage{fancyhdr}
\author{\small\scshape S\lowercase{teven} D\lowercase{uplij}}

\address{
Center for Information Technology (WWU IT),
University of M\"unster,
R\"ontgenstrasse 7-13\\
D-48149 M\"unster,
Deutschland}
\email{\small \sf douplii@uni-muenster.de;
sduplij@gmail.com;
https://ivv5hpp.uni-muenster.de/u/douplii}
\title{\large\bfseries\scshape
P\lowercase{olyadic rings of $p$-adic integers}}

\date{\textit{of start} October 10, 2022. \textit{Date}:
\textit{of completion}
November 1, 2022.
\textit{Version 2}: November 27, 2022.
\newline
\mbox{}\hskip 1.16em
\textit{Total}:
18
references.
}

\renewcommand{\refname}{\textsc{References}}

\let\origsection\section
\renewcommand{\section}[1]{\sectionmark{#1}\origsection{#1}}
\let\origsubsection\subsection
\renewcommand{\subsection}[1]{\subsectionmark{#1}\origsubsection{#1}}

\makeatletter
\renewenvironment{thebibliography}[1]{%
  \@xp\origsection\@xp*\@xp{\refname}%
  \normalfont\footnotesize\labelsep .9em\relax
  \renewcommand\theenumiv{\arabic{enumiv}}\let\p@enumiv\@empty
  \vspace*{-5pt}
  \list{\@biblabel{\theenumiv}}{\settowidth\labelwidth{\@biblabel{#1}}%
    \leftmargin\labelwidth \advance\leftmargin\labelsep
    \usecounter{enumiv}}%
  \sloppy \clubpenalty\@M \widowpenalty\clubpenalty
  \sfcode`\.=\@m
}{%
  \def\@noitemerr{\@latex@warning{Empty `thebibliography' environment}}%
  \endlist
}
\makeatother


\begin{document}
\mbox{}
\vspace{1cm}

\mbox{}
\begin{abstract}

\noindent  In this note we, first, recall that the sets of all representatives
of some special ordinary residue classes become $\left(  m,n\right)  $-rings.
Second, we introduce a possible $p$-adic analog of the residue class modulo a
$p$-adic integer. Then, we find the relations which determine, when the
representatives form a $\left(  m,n\right)  $-ring.
At the very short spacetime scales such rings could lead to new symmetries of modern particle models.

\end{abstract}
\maketitle

\thispagestyle{empty}


\mbox{}
\vspace{1.5cm}

\tableofcontents
\newpage

\pagestyle{fancy}

\addtolength{\footskip}{15pt}

\renewcommand{\sectionmark}[1]{%
\markboth{
{ \scshape #1}}{}}

\renewcommand{\subsectionmark}[1]{%
\markright{
\mbox{\;}\\[5pt]
\textmd{#1}}{}}

\fancyhead{}
\fancyhead[EL,OR]{\leftmark}
\fancyhead[ER,OL]{\rightmark}
\fancyfoot[C]{\scshape -- \textcolor{BrickRed}{\thepage} --}

\renewcommand\headrulewidth{0.5pt}
\fancypagestyle {plain1}{ %
\fancyhf{}
\renewcommand {\headrulewidth }{0pt}
\renewcommand {\footrulewidth }{0pt}
}

\fancypagestyle{plain}{ %
\fancyhf{}
\fancyhead[C]{\scshape S\lowercase{teven} D\lowercase{uplij} \hskip 0.7cm \MakeUppercase{Polyadic Hopf algebras and quantum groups}}
\fancyfoot[C]{\scshape - \thepage  -}
\renewcommand {\headrulewidth }{0pt}
\renewcommand {\footrulewidth }{0pt}
}

\fancypagestyle{fancyref}{ %
\fancyhf{} 
\fancyhead[C]{\scshape R\lowercase{eferences} }
\fancyfoot[C]{\scshape -- \textcolor{BrickRed}{\thepage} --}
\renewcommand {\headrulewidth }{0.5pt}
\renewcommand {\footrulewidth }{0pt}
}

\fancypagestyle{emptyf}{
\fancyhead{}
\fancyfoot[C]{\scshape -- \textcolor{BrickRed}{\thepage} --}
\renewcommand{\headrulewidth}{0pt}
}
\mbox{}
\thispagestyle{emptyf}

\section{\textsc{Introduction}}

The fundamental conception of $p$-adic numbers is based on a special extension
of rational numbers that is an alternative to real and complex
numbers. The main idea is the completion of the rational numbers with respect
to the $p$-adic norm, which is non-Archimedean. Nowadays, $p$-adic methods are
widely used in number theory \cite{neurkich,samuel}, arithmetic geometry
\cite{ber/ogu,stum} and algorithmic computations \cite{car2017}. In
mathematical physics, a non-Archimedean approach to spacetime and string
dynamics at the Planck scale leads to new symmetries of particle models (see,
e.g., \cite{dra/khr/koz,vla/vol/zel} and the references therein). For~some special
applications, see, e.g., \cite{sar/hum/hua,zad/sur/tun}. General reviews are
given in \cite{koblitz,robert,schikhof}.

Previously, we have studied the algebraic structure of the representative set
in a fixed ordinary residue class \cite{dup2017a}. We found that the set of
representatives becomes a polyadic or $\left(  m,n\right)  $-ring, if the
parameters of a class satisfy special ``quantization'' conditions. We have found
that similar polyadic structures appear naturally for $p$-adic integers, if we
introduce informally a $p$-adic analog of the residue classes, and we investigate
here the set of its representatives along the lines of
\cite{dup2017a,dup2019,duplij2022}.

\section{\label{sec-res}$\left(  m,n\right)  $-\textsc{rings of integer
numbers from residue classes}}

Here we recall that representatives of special residue (congruence) classes
can form polyadic rings, as was found in \cite{dup2017a,dup2019} (see also
notation from \cite{duplij2022}).

Let us denote the residue (congruence) class of an integer $a$ modulo $b$ by%
\begin{equation}
\left[  a\right]  _{b}=\left\{  \left\{  r_{k}\left(  a,b\right)  \right\}
\mid k\in\mathbb{Z},a\in\mathbb{Z}_{+},b\in\mathbb{N},0\leq a\leq b-1\right\}
, \label{ab}%
\end{equation}
where $r_{k}\left(  a,b\right)  =a+bk$ is a generic representative element of
the class $\left[  a\right]  _{b}$. The~canonical representative is the smallest
non-negative number of these. Informally, $a$ is the remainder of
$r_{k}\left(  a,b\right)  $ when divided by $b$. The corresponding equivalence
relation (congruence modulo $b$) is denoted by%
\begin{equation}
r=a\left(  \operatorname{mod}b\right)  . \label{ra}%
\end{equation}

Introducing the binary operations between classes $\left(  +_{cl},\times
_{cl}\right)  $, the addition $\left[  a_{1}\right]  _{b}+_{cl}\left[
a_{2}\right]  _{b}=\left[  a_{1}+a_{2}\right]  _{b}$, and the multiplication
$\left[  a_{1}\right]  _{b}\times_{cl}\left[  a_{2}\right]  _{b}=\left[
a_{1}a_{2}\right]  _{b}$, the residue class (binary) finite commutative ring
$\mathbb{Z}\diagup b\mathbb{Z}$ (with identity) is defined in the standard method
(which was named \textquotedblleft external\textquotedblright\ \cite{dup2017a}%
). If $a\neq0$ and $b$ is prime, then $\mathbb{Z}\diagup b\mathbb{Z}$ becomes
a finite field.

The set of representatives $\left\{  r_{k}\left(  a,b\right)  \right\}  $ in a
given class $\left[  a\right]  _{b}$ does not form a binary ring, because
there are no binary operations (addition and multiplication) which are
simultaneously closed for arbitrary $a$ and $b$. Nevertheless, the following
polyadic operations on representatives $r_{k}=r_{k}\left(  a,b\right)  $,
$m$-ary addition $\nu_{m}$%
\begin{equation}
\nu_{m}\left[  r_{k_{1}},r_{k_{2}},\ldots,r_{k_{m}}\right]  =r_{k_{1}%
}+r_{k_{2}}+\ldots+r_{k_{m}}, \label{nu}%
\end{equation}
and $n$-ary multiplication $\mu_{n}$%
\begin{equation}
\mu_{n}\left[  r_{k_{1}},r_{k_{2}},\ldots,r_{k_{n}}\right]  =r_{k_{1}}%
r_{k_{2}}\ldots r_{k_{n}},\ \ \ r_{k_{i}}\in\left[  a\right]  _{b},\ k_{i}%
\in\mathbb{Z}, \label{mu}%
\end{equation}
can be closed but~only for special values of $a=a_{q}$ and $b=b_{q}$, which
defines the nonderived $\left(  m,n\right)  $-ary ring%
\begin{equation}
\mathbb{Z}_{\left(  m,n\right)  }\left(  a_{q},b_{q}\right)  =\left\langle
\left[  a_{q}\right]  _{b_{q}}\mid\nu_{m},\mu_{n}\right\rangle \label{z}%
\end{equation}
of polyadic integers (that was called the \textquotedblleft
internal\textquotedblright\ way \cite{dup2017a}). The conditions of closure
for the operations between representatives can be formulated in terms of the
(arity shape~\cite{dup2019}) invariants (which may be seen as some form of
\textquotedblleft quantization\textquotedblright)%
\begin{align}
\left(  m-1\right)  \frac{a_{q}}{b_{q}}  &  =I_{m}\left(  a_{q},b_{q}\right)
\in\mathbb{N},\label{an1}\\
a_{q}^{n-1}\frac{a_{q}-1}{b_{q}}  &  =J_{n}\left(  a_{q},b_{q}\right)
\in\mathbb{N}, \label{an2}%
\end{align}
or, equivalently, using the congruence relations \cite{dup2017a}%
\begin{align}
ma_{q}  &  \equiv a_{q}\left(  \operatorname{mod}b_{q}\right)  ,\label{am1}\\
a_{q}^{n}  &  \equiv a_{q}\left(  \operatorname{mod}b_{q}\right)  ,
\label{am2}%
\end{align}
where we have denoted by $a_{q}$ and $b_{q}$ the concrete solutions of the
\textquotedblleft quantization\textquotedblright\ equations (\ref{an1}%
)--(\ref{am2}). To understand the nature of the \textquotedblleft
quantization\textquotedblright, we consider in detail the concrete example of
nonderived $m$-ary addition and $n$-ary multiplication appearance for
representatives in a fixed residue class.

\begin{example}
Let us consider the following residue class%
\begin{equation}
\left[  \left[  3\right]  \right]  _{4}=\left\{  \ldots
-45,-33,-29,-25,-21,-17,-13,-9,-5,-1,3,7,11,15,19,23,27,31\ldots\right\}  ,
\label{34}%
\end{equation}
where the representatives are%
\begin{equation}
r_{k}=r_{k}\left(  3,4\right)  =3+4k,\ \ \ \ k\in\mathbb{Z}. \label{r34}%
\end{equation}

We first obtain the condition for when the sum of $m$ representatives belongs to
the class (\ref{34}). So, we compute step by step%
\begin{align}
m  &  =2,\ \ \ r_{k_{1}}+r_{k_{2}}=r_{k}+3\notin\left[  \left[  3\right]
\right]  _{4},\ \ \ k=k_{1}+k_{2},\label{r5}\\
m  &  =3,\ \ \ r_{k_{1}}+r_{k_{2}}+r_{k_{3}}=r_{k}+6\notin\left[  \left[
3\right]  \right]  _{4},\ \ \ k=k_{1}+k_{2}+k_{3},\\
m  &  =4,\ \ \ r_{k_{1}}+r_{k_{2}}+r_{k_{3}}+r_{k_{4}}=r_{k}+9\notin\left[
\left[  3\right]  \right]  _{4},\ \ \ k=k_{1}+k_{2}+k_{3}+k_{4},\\
m  &  =5,\ \ \ r_{k_{1}}+r_{k_{2}}+r_{k_{3}}+r_{k_{4}}+r_{k_{5}}%
=\fbox{r$_{k}\in\left[  \left[  3\right]  \right]  _{4}$},\ \ \ k=k_{1}%
+k_{2}+k_{3}+k_{4}+k_{5}+3.\label{m5}\\
&  \vdots\nonumber
\end{align}
Thus, the binary, ternary and $4$-ary additions are not closed, while $5$-ary
addition is. In general, the closure of $m$-ary addition holds valid when
$4\mid\left(  m-1\right)  $, that is for $m=5,9,13,17,\ldots$, such that%
\begin{equation}
\sum_{i=1}^{m}r_{k_{i}}=r_{k}\in\left[  \left[  3\right]  \right]
_{4},\ \ \ k=\sum_{i=1}^{m}k_{i}+3\ell_{\nu},
\end{equation}
where $\ell_{\nu}=\frac{m-1}{4}\in\mathbb{N}$ is a natural number. The
\textquotedblleft quantization\textquotedblright\ rule for the arity of
addition (\ref{an1}), (\ref{am1}) becomes $m=4\ell_{\nu}+1$.

If we consider the minimal arity $m=5$, we arrive at the conclusion that
$\left\langle \left\{  r_{k}\right\}  \mid\nu_{5}\right\rangle $, is a $5$-ary
commutative semigroup, where $\nu_{5}$ is the nonderived (i.e. not composed from
lower arity operations) $5$-ary addition%
\begin{equation}
\nu_{5}\left[  r_{k_{1}},r_{k_{2}},r_{k_{3}},r_{k_{4}},r_{k_{5}}\right]
=r_{k_{1}}+r_{k_{2}}+r_{k_{3}}+r_{k_{4}}+r_{k_{5}},\ \ \ r_{k_{i}}\in\left[
3\right]  _{4}, \label{nu1}%
\end{equation}
given by (\ref{m5}), and total $5$-ary associativity follows from that of
the binary addition in (\ref{nu}) and (\ref{nu1}). In this case $\ell_{\nu}$
is the \textquotedblleft number\textquotedblright\ of composed $5$-ary
additions (the polyadic power). There is no neutral element $z$ for $5$-ary
addition $\nu_{5}$ (\ref{nu1}) defined by $\nu_{5}\left[  z,z,z,z,r_{k}%
\right]  =r_{k}$, and so the semigroup $\left\langle \left\{  r_{k}\right\}
\mid\nu_{5}\right\rangle $ is zeroless. Nevertheless, $\left\langle \left\{
r_{k}\right\}  \mid\nu_{5}\right\rangle $ is a $5$-ary group (which is
impossible in the binary case, where all groups contain a neutral element, the
identity), because each element $r_{k}$ has a unique querelement
$\widetilde{r_{k}}$ defined by (see,e.g., \cite{duplij2022})%
\begin{equation}
\nu_{5}\left[  r_{k},r_{k},r_{k},r_{k},\widetilde{r_{k}}\right]
=r_{k},\ \ \ r_{k},\widetilde{r_{k}}\in\left[  3_{4}\right]  . \label{vr}%
\end{equation}
Therefore, from (\ref{m5}) and (\ref{nu1}) it follows that $\tilde{r}%
_{k}=-3r_{k}=r_{-9-12k}$. For instance, for the first several elements of the
residue class $\left[  3\right]  _{4}$ (\ref{34}), we have the following
querelements%
\begin{align}
\widetilde{7}  &  =-21,\ \ \widetilde{11}=-33,\ \ \widetilde{15}=-45,\\
\widetilde{-1}  &  =3,\ \ \widetilde{-5}=15,\ \ \widetilde{-9}=27.
\end{align}

Note that in the $5$-ary group $\left\langle \left\{  r_{k}\right\}  \mid\nu
_{5},\widetilde{\left(  \cdot\right)  }\right\rangle $, the additive
quermapping defined by $r_{k}\mapsto\widetilde{r_{k}}$ is not a reflection
(of any order) for $m\geq3$, i.e.,~$\widetilde{\widetilde{r_{k}}}\neq r_{k}$
(as opposed to the inverse in the binary case) \cite{duplij2022}.

Now, we turn to the multiplication of $n$ representatives (\ref{r34}) of the
residue class $\left[  3\right]  _{4}$ (\ref{34}). By analogy with
(\ref{r5})--(\ref{m5}) we obtain, step by step%
\begin{equation}%
\begin{array}
[c]{ll}%
n=2 & r_{k_{1}}r_{k_{2}}=r_{k}+6\notin\left[  \left[  3\right]  \right]
_{4},\ \ \ k=3k_{1}+3k_{2}+4k_{1}k_{2},\\[5pt]%
n=3 & \left\{
\begin{array}
[c]{l}%
r_{k_{1}}r_{k_{2}}r_{k_{3}}=\fbox{r$_{k}\in\left[  \left[  3\right]  \right]
_{4}$},\\
k=9k_{1}+9k_{2}+9k_{3}+12k_{1}k_{2}+12k_{1}k_{3}+12k_{2}k_{3}+16k_{1}%
k_{2}k_{3}+6,
\end{array}
\right. \\[10pt]%
n=4 & \left\{
\begin{array}
[c]{l}%
r_{k_{1}}r_{k_{2}}r_{k_{3}}r_{k_{4}}=r_{k}+2\notin\left[  \left[  3\right]
\right]  _{4},\\
k=27k_{1}+27k_{2}+27k_{3}+27k_{4}+36k_{1}k_{2}+36k_{1}k_{3}+36k_{1}k_{4}\\
+36k_{2}k_{3}+36k_{2}k_{4}+36k_{3}k_{4}+48k_{1}k_{2}k_{3}+48k_{1}k_{2}%
k_{4}+48k_{1}k_{3}k_{4}\\
+48k_{2}k_{3}k_{4}+64k_{1}k_{2}k_{3}k_{4}+19,
\end{array}
\right. \\[15pt]%
n=5 & \left\{
\begin{array}
[c]{l}%
r_{k_{1}}r_{k_{2}}r_{k_{3}}r_{k_{4}}r_{k_{5}}=\fbox{r$_{k}\in\left[  \left[
3\right]  \right]  _{4}$},\\
k=81k_{1}+81k_{2}+81k_{3}+81k_{4}+81k_{5}+108k_{1}k_{2}+108k_{1}k_{3}%
+108k_{1}k_{4}\\
+108k_{2}k_{3}+108k_{1}k_{5}+108k_{2}k_{4}+108k_{2}k_{5}+108k_{3}%
k_{4}+108k_{3}k_{5}+108k_{4}k_{5}\\
+144k_{1}k_{2}k_{3}+144k_{1}k_{2}k_{4}+144k_{1}k_{2}k_{5}+144k_{1}k_{3}%
k_{4}+144k_{1}k_{3}k_{5}+144k_{2}k_{3}k_{4}\\
+144k_{1}k_{4}k_{5}+144k_{2}k_{3}k_{5}+144k_{2}k_{4}k_{5}+144k_{3}k_{4}%
k_{5}+192k_{1}k_{2}k_{3}k_{4}+192k_{1}k_{2}k_{3}k_{5}\\
+192k_{1}k_{2}k_{4}k_{5}+192k_{1}k_{3}k_{4}k_{5}+192k_{2}k_{3}k_{4}%
k_{5}+256k_{1}k_{2}k_{3}k_{4}k_{5}+60.
\end{array}
\right.
\end{array}
\label{n5}%
\end{equation}

By direct computation, we observe that the binary and $4$-ary multiplications
are not closed, but the ternary and $5$-ary ones are closed. In general, the
product of $n=2\ell_{\mu}+1$ ($\ell_{\mu}\in\mathbb{N}$) elements of the
residue class $\left[  3\right]  _{4}$ is in the class, which is the
\textquotedblleft quantization\textquotedblright\ rule for multiplication
(\ref{an2}) and (\ref{am2}). Again we consider the minimal arity $n=3$ of
multiplication and observe that $\left\langle \left\{  r_{k}\right\}  \mid
\mu_{3}\right\rangle $ is a commutative ternary semigroup, where%
\begin{equation}
\mu_{3}\left[  r_{k_{1}},r_{k_{2}},r_{k_{3}}\right]  =r_{k_{1}}r_{k_{2}%
}r_{k_{3}},\ \ \ \ r_{k_{i}}\in\left[  3\right]  _{4} \label{mu1}%
\end{equation}
is a nonderived ternary multiplication, and the total ternary associativity
of $\mu_{3}$ is governed by associativity of the binary product in (\ref{mu})
and (\ref{mu1}). As opposed to the $5$-ary addition, $\left\langle \left\{
r_{k}\right\}  \mid\mu_{3}\right\rangle $ is not a group, because not all
elements have a unique querelement. However, a polyadic identity $e$
defined by (i.e. as a neutral element of the ternary multiplication $\mu_{3}$)%
\begin{equation}
\mu_{3}\left[  e,e,r_{k}\right]  =r_{k},\ \ e,r_{k}\in\left[  3\right]  _{4},
\label{e}%
\end{equation}
exists and is equal to $e=-1$.

The polyadic distributivity between $\nu_{5}$ and $\mu_{3}$ \cite{duplij2022}
follows from the binary distributivity in $\mathbb{Z}$ and (\ref{nu1}%
),(\ref{mu1}), and therefore the residue class $\left[  3\right]  _{4}$ has
the algebraic structure of the polyadic ring%
\begin{equation}
\mathbb{Z}_{\left(  5,3\right)  }=\mathbb{Z}_{\left(  5,3\right)  }\left(
3,4\right)  =\left\langle \left\{  r_{k}\right\}  \mid\nu_{5},\widetilde
{\left(  \cdot\right)  },\mu_{3},e\right\rangle ,\ \ \ e,r_{k}\in\left[
3\right]  _{4},
\end{equation}
which is a commutative zeroless $\left(  5,3\right)  $-ring with the additive
quermapping $\widetilde{\left(  \cdot\right)  }$ (\ref{vr}) and the
multiplicative neutral element $e$ (\ref{e}).
\end{example}

The arity shape of the ring of polyadic integers $\mathbb{Z}_{\left(
m,n\right)  }\left(  a_{q},b_{q}\right)  $ (\ref{z}) is the (surjective)
mapping%
\begin{equation}
\left(  a_{q},b_{q}\right)  \Longrightarrow\left(  m,n\right)  . \label{aq}%
\end{equation}

The mapping (\ref{aq}) for the lowest values of $a_{q},b_{q}$ is given in
\textsc{Table} \ref{T1} ($I=I_{m}\left(  a_{q},b_{q}\right)  $, $J=J_{n}%
\left(  a_{q},b_{q}\right)  $).
\begin{table}[H]
\caption{The arity shape mapping (\ref{aq}) for the polyadic ring
$\mathbb{Z}_{\left(  m,n\right)  }\left(  a_{q},b_{q}\right)  $ (\ref{z}).
Empty cells indicate that no such structure exists.  }
\begin{center}
{\tiny
\begin{tabular}
[c]{||c||c|c|c|c|c|c|c|c|c||}\hline\hline
$a_{q}\setminus b_{q}$ & 2 & 3 & 4 & 5 & 6 & 7 & 8 & 9 & 10\\\hline\hline
1 & $%
\begin{array}
[c]{c}%
m=\mathbf{3}\\
n=\mathbf{2}\\
I=1\\
J=0
\end{array}
$ & $%
\begin{array}
[c]{c}%
m=\mathbf{4}\\
n=\mathbf{2}\\
I=1\\
J=0
\end{array}
$ & $%
\begin{array}
[c]{c}%
m=\mathbf{5}\\
n=\mathbf{2}\\
I=1\\
J=0
\end{array}
$ & $%
\begin{array}
[c]{c}%
m=\mathbf{6}\\
n=\mathbf{2}\\
I=1\\
J=0
\end{array}
$ & $%
\begin{array}
[c]{c}%
m=\mathbf{7}\\
n=\mathbf{2}\\
I=1\\
J=0
\end{array}
$ & $%
\begin{array}
[c]{c}%
m=\mathbf{8}\\
n=\mathbf{2}\\
I=1\\
J=0
\end{array}
$ & $%
\begin{array}
[c]{c}%
m=\mathbf{9}\\
n=\mathbf{2}\\
I=1\\
J=0
\end{array}
$ & $%
\begin{array}
[c]{c}%
m=\mathbf{10}\\
n=\mathbf{2}\\
I=1\\
J=0
\end{array}
$ & $%
\begin{array}
[c]{c}%
m=\mathbf{11}\\
n=\mathbf{2}\\
I=1\\
J=0
\end{array}
$\\\hline
2 &  & $%
\begin{array}
[c]{c}%
m=\mathbf{4}\\
n=\mathbf{3}\\
I=2\\
J=2
\end{array}
$ &  & $%
\begin{array}
[c]{c}%
m=\mathbf{6}\\
n=\mathbf{5}\\
I=2\\
J=6
\end{array}
$ & $%
\begin{array}
[c]{c}%
m=\mathbf{4}\\
n=\mathbf{3}\\
I=1\\
J=1
\end{array}
$ & $%
\begin{array}
[c]{c}%
m=\mathbf{8}\\
n=\mathbf{4}\\
I=2\\
J=2
\end{array}
$ &  & $%
\begin{array}
[c]{c}%
m=\mathbf{10}\\
n=\mathbf{7}\\
I=2\\
J=14
\end{array}
$ & $%
\begin{array}
[c]{c}%
m=\mathbf{6}\\
n=\mathbf{5}\\
I=1\\
J=3
\end{array}
$\\\hline
3 &  &  & $%
\begin{array}
[c]{c}%
m=\mathbf{5}\\
n=\mathbf{3}\\
I=3\\
J=6
\end{array}
$ & $%
\begin{array}
[c]{c}%
m=\mathbf{6}\\
n=\mathbf{5}\\
I=3\\
J=48
\end{array}
$ & $%
\begin{array}
[c]{c}%
m=\mathbf{3}\\
n=\mathbf{2}\\
I=1\\
J=1
\end{array}
$ & $%
\begin{array}
[c]{c}%
m=\mathbf{8}\\
n=\mathbf{7}\\
I=3\\
J=312
\end{array}
$ & $%
\begin{array}
[c]{c}%
m=\mathbf{9}\\
n=\mathbf{3}\\
I=3\\
J=3
\end{array}
$ &  & $%
\begin{array}
[c]{c}%
m=\mathbf{11}\\
n=\mathbf{5}\\
I=3\\
J=24
\end{array}
$\\\hline
4 &  &  &  & $%
\begin{array}
[c]{c}%
m=\mathbf{6}\\
n=\mathbf{3}\\
I=4\\
J=12
\end{array}
$ & $%
\begin{array}
[c]{c}%
m=\mathbf{4}\\
n=\mathbf{2}\\
I=2\\
J=2
\end{array}
$ & $%
\begin{array}
[c]{c}%
m=\mathbf{8}\\
n=\mathbf{4}\\
I=4\\
J=36
\end{array}
$ &  & $%
\begin{array}
[c]{c}%
m=\mathbf{10}\\
n=\mathbf{4}\\
I=4\\
J=28
\end{array}
$ & $%
\begin{array}
[c]{c}%
m=\mathbf{6}\\
n=\mathbf{3}\\
I=2\\
J=6
\end{array}
$\\\hline
5 &  &  &  &  & $%
\begin{array}
[c]{c}%
m=\mathbf{7}\\
n=\mathbf{3}\\
I=5\\
J=20
\end{array}
$ & $%
\begin{array}
[c]{c}%
m=\mathbf{8}\\
n=\mathbf{7}\\
I=11\\
J=11160
\end{array}
$ & $%
\begin{array}
[c]{c}%
m=\mathbf{9}\\
n=\mathbf{3}\\
I=5\\
J=15
\end{array}
$ & $%
\begin{array}
[c]{c}%
m=\mathbf{10}\\
n=\mathbf{7}\\
I=5\\
J=8680
\end{array}
$ & $%
\begin{array}
[c]{c}%
m=\mathbf{3}\\
n=\mathbf{2}\\
I=1\\
J=2
\end{array}
$\\\hline
6 &  &  &  &  &  & $%
\begin{array}
[c]{c}%
m=\mathbf{8}\\
n=\mathbf{3}\\
I=6\\
J=30
\end{array}
$ &  &  & $%
\begin{array}
[c]{c}%
m=\mathbf{6}\\
n=\mathbf{2}\\
I=3\\
J=3
\end{array}
$\\\hline
7 &  &  &  &  &  &  & $%
\begin{array}
[c]{c}%
m=\mathbf{9}\\
n=\mathbf{3}\\
I=7\\
J=42
\end{array}
$ & $%
\begin{array}
[c]{c}%
m=\mathbf{10}\\
n=\mathbf{4}\\
I=7\\
J=266
\end{array}
$ & $%
\begin{array}
[c]{c}%
m=\mathbf{11}\\
n=\mathbf{5}\\
I=7\\
J=1680
\end{array}
$\\\hline
8 &  &  &  &  &  &  &  & $%
\begin{array}
[c]{c}%
m=\mathbf{10}\\
n=\mathbf{3}\\
I=8\\
J=56
\end{array}
$ & $%
\begin{array}
[c]{c}%
m=\mathbf{6}\\
n=\mathbf{5}\\
I=4\\
J=3276
\end{array}
$\\\hline
9 &  &  &  &  &  &  &  &  & $%
\begin{array}
[c]{c}%
m=\mathbf{11}\\
n=\mathbf{3}\\
I=9\\
J=72
\end{array}
$\\\hline\hline
\end{tabular}
}
\end{center}
\label{T1}%
\end{table}

The binary ring of ordinary integers $\mathbb{Z}$ corresponds to $\left(
a_{q}=0,b_{q}=1\right)  \Longrightarrow\left(  2,2\right)  $ or $\mathbb{Z=Z}%
_{\left(  2,2\right)  }\left(  0,1\right)  $, $I=J=0$.

\section{\textsc{Representations of }$p$\textsc{-adic integers}}

Let us explore briefly some well-known definitions regarding $p$-adic integers
to establish notations (for reviews, see \cite{gouvea,koblitz,robert}).

A $p$-adic integer is an infinite formal sum of the form%
\begin{equation}
\mathbf{x}=\mathbf{x}\left(  p\right)  =\alpha_{0}+\alpha_{1}p+\alpha_{2}%
p^{2}+\ldots+\alpha_{i-1}p^{i-1}+\alpha_{i}p^{i}+\alpha_{i+1}p^{i+1}%
+\ldots,\ \ \ \alpha_{i}\in\mathbb{Z}, \label{x}%
\end{equation}
where the digits (denoted by Greek letters from the beginning of alphabet)
$0\leq\alpha_{i}\leq p-1$, and $p\geq2$ is a fixed prime number. The expansion
(\ref{x}) is called standard (or canonical), and $\alpha_{i}$ are the $p$-adic
digits which are usually written from the right to the left (positional
notation) $\mathbf{x=.\ }...\alpha_{i+1}\alpha_{i}\alpha_{i-1}\ldots\alpha
_{2}\alpha_{1}\alpha_{0}$ or sometimes $\mathbf{x=}\left\{  \alpha_{0}%
,\alpha_{1},\alpha_{2},\ldots,\alpha_{i-1},\alpha_{i},\alpha_{i+1}%
\ldots\right\}  $. The set of $p$-adic integers is a commutative ring (of
$p$-adic integers) denoted by $\mathbb{Z}_{p}=\left\{  \mathbf{x}\right\}  $,
and the ring of ordinary integers (sometimes called \textquotedblleft
rational\textquotedblright\ integers) $\mathbb{Z}$ is its (binary) subring.

The so called coherent representation of $\mathbb{Z}_{p}$ is based on the
(inverse) projective limit of finite fields $\mathbb{Z}\diagup p^{l}%
\mathbb{Z}$, because the surjective map $\mathbb{Z}_{p}\longrightarrow
\mathbb{Z}\diagup p^{l}\mathbb{Z}$ defined by%
\begin{equation}
\alpha_{0}+\alpha_{1}p+\alpha_{2}p^{2}+\ldots+\alpha_{i}p^{i}+\ldots
\mapsto\left(  \alpha_{0}+\alpha_{1}p+\alpha_{2}p^{2}+\ldots+\alpha
_{l-1}p^{l-1}\right)  \operatorname{mod}p^{l}%
\end{equation}
is a ring homomorphism. In this case, a $p$-adic integer is the infinite
Cauchy sequence that converges to%
\begin{equation}
\mathbf{x}=\mathbf{x}\left(  p\right)  =\left\{  x_{i}\left(  p\right)
\right\}  _{i=1}^{\infty}=\left\{  x_{1}\left(  p\right)  ,x_{2}\left(
p\right)  ,\ldots,x_{i}\left(  p\right)  \ldots\right\}  , \label{xc}%
\end{equation}
where%
\begin{equation}
x_{i}\left(  p\right)  =\alpha_{0}+\alpha_{1}p+\alpha_{2}p^{2}+\ldots
+\alpha_{l-1}p^{l-1}%
\end{equation}
with the coherency condition%
\begin{equation}
x_{i+1}\left(  p\right)  \equiv x_{i}\left(  p\right)  \operatorname{mod}%
p^{i},\ \ \ \forall i\geq1, \label{cog}%
\end{equation}
and the $p$-adic digits are $0\leq\alpha_{i}\leq p-1$.

If $0\leq x_{i}\left(  p\right)  \leq p^{i}-1$ for all $i\geq1$, then the
coherent representation (\ref{xc}) is said to be reduced. The~ordinary integers
$x\in\mathbb{Z}$ embed into $p$-adic integers as constant infinite sequences
by $x\mapsto\left\{  x,x,\ldots,x,\ldots\right\}  $.

Using the fact that the process of reducing modulo $p^{i}$ is equivalent to
vanishing the last $i$ digits, the coherency condition (\ref{cog}) leads to a
sequence of partial sums \cite{gouvea}%
\begin{equation}
\mathbf{x}=\mathbf{x}\left(  p\right)  =\left\{  y_{i}\left(  p\right)
\right\}  _{i=1}^{\infty}=\left\{  y_{1}\left(  p\right)  ,y_{2}\left(
p\right)  ,\ldots,y_{i}\left(  p\right)  \ldots\right\}  , \label{xy}%
\end{equation}
where%
\begin{equation}
y_{1}\left(  p\right)  =\alpha_{0},\ y_{2}\left(  p\right)  =\alpha_{0}%
+\alpha_{1}p,\ y_{3}\left(  p\right)  =\alpha_{0}+\alpha_{1}p+\alpha_{2}%
p^{2},\ y_{4}\left(  p\right)  =\alpha_{0}+\alpha_{1}p+\alpha_{2}p^{2}%
+\alpha_{3}p^{3},\ \ldots\ . \label{y}%
\end{equation}

Sometimes, the partial sum representation (\ref{xy}) is simpler for $p$-adic
integer computations.

\section{$\left(  m,n\right)  $\textsc{-rings of }$p$\textsc{-adic integers}}

As may be seen from \textsc{Section \ref{sec-res}} and \cite{dup2017a,dup2019}%
, the construction of the nonderived $\left(  m,n\right)  $-rings of ordinary
(\textquotedblleft rational\textquotedblright) integers $\mathbb{Z}_{\left(
m,n\right)  }\left(  a_{q},b_{q}\right)  $ (\ref{z}) can be performed in terms of
residue class representatives (\ref{ab}). To introduce a $p$-adic analog of
the residue class (\ref{ab}), one needs some ordering concept, which does not
exist for $p$-adic integers \cite{gouvea}. Nevertheless, one could informally
define the following analog of ordering.

\begin{definition}
A \textquotedblleft componentwise strict order\textquotedblright\ $<_{comp}$
is a multicomponent binary relation between $p$-adic numbers $\mathbf{a}%
=\left\{  \alpha_{i}\right\}  _{i=0}^{\infty}$, $0\leq\alpha_{i}\leq p-1$ and
$\mathbf{b}=\left\{  \beta_{i}\right\}  _{i=0}^{\infty}$, $0\leq\beta_{i}\leq
p-1$, such that%
\begin{equation}
\mathbf{a}<_{comp}\mathbf{b}\Longleftrightarrow\alpha_{i}<\beta_{i}%
,\ \ \ \text{for\ all}\ i=0,\ldots,\infty,\ \ \mathbf{a},\mathbf{b}%
\in\mathbb{Z}_{p},\ \ \alpha_{i},\beta_{i}\in\mathbb{Z}.
\end{equation}

A \textquotedblleft componentwise nonstrict order\textquotedblright%
\ $\leq_{comp}$ is defined in the same way, but using the nonstrict order
$\leq$ for component integers from $\mathbb{Z}$ (digits).
\end{definition}

Using this definition, we can define a $p$-adic analog of the residue class
informally by changing $\mathbb{Z}$ to $\mathbb{Z}_{p}$ in (\ref{ab}).

\begin{definition}
A $p$-adic analog of the residue class of $\mathbf{a}$ modulo $\mathbf{b}$ is%
\begin{equation}
\left[  \mathbf{a}\right]  _{\mathbf{b}}=\left\{  \left\{  \mathbf{r}%
_{\mathbf{k}}\left(  \mathbf{a},\mathbf{b}\right)  \right\}  \mid
\mathbf{a},\mathbf{b},\mathbf{k}\in\mathbb{Z}_{p},0\leq\mathbf{a}%
<\mathbf{b}\right\}  , \label{ab1}%
\end{equation}
and the generic representative of the class is%
\begin{equation}
\mathbf{r}_{\mathbf{k}}\left(  \mathbf{a},\mathbf{b}\right)  =\mathbf{a}%
+_{p}\mathbf{b\bullet}_{p}\mathbf{k}, \label{rk}%
\end{equation}
where $+_{p}$ and $\mathbf{\bullet}_{p}$ are the binary sum and the binary
product of $p$-adic integers (we treat them componentwise in the partial sum
representation (\ref{y})), and~the $i$th component of (\ref{rk})'s right hand side is
computed by $\operatorname{mod}p^{i}$.
\end{definition}

As with the ordinary (\textquotedblleft rational\textquotedblright) integers
(\ref{ab}), the~$p$-adic integer $\mathbf{a}$ can be treated as a type of
remainder for the representative $\mathbf{r}_{\mathbf{k}}\left(
\mathbf{a},\mathbf{b}\right)  $ when divided by the $p$-adic integer
$\mathbf{b}$. We denote the corresponding $p$-adic analog of (\ref{ra})
(informally, a $p$-adic analog of the congruence modulo $\mathbf{b}$) as%
\begin{equation}
\mathbf{r}=\mathbf{a}\left(  \operatorname{Mod}_{p}\mathbf{b}\right)  .
\label{ra1}%
\end{equation}

\begin{remark}
In general, to build a nonderived $\left(  m,n\right)  $-ring along the lines
of \textsc{Section \ref{sec-res}}, we do not need any analog of the residue
class at all, but only the concrete form of the representative (\ref{rk}).
Then demanding the closure of $m$-ary addition (\ref{nu}) and $n$-ary
multiplication (\ref{mu}), we obtain conditions on the parameters (now digits
of $p$-adic integers) similar to (\ref{an1}) and (\ref{an2}).
\end{remark}

In the partial sum representation (\ref{xy}), the case of ordinary
(\textquotedblleft rational\textquotedblright) integers corresponds to the
first component (first digit $\alpha_{0}$) of the $p$-adic integer (\ref{y}),
and higher components can be computed using the explicit formulas for sum and
product of $p$-adic integers \cite{xu/dai}. Because they are too cumbersome,
we present here the \textquotedblleft block-schemes\textquotedblright\ of the
computations, while concrete examples can be obtained componentwise using
(\ref{y}).

\begin{lemma}
The $p$-adic analog of the residue class (\ref{ab1}) is a commutative $m$-ary
group $\left\langle \left[  \mathbf{a}\right]  _{\mathbf{b}}\mid\mathbf{\nu
}_{m}\right\rangle $, if%
\begin{equation}
\left(  m-1\right)  \mathbf{a}=\mathbf{b\bullet}_{p}\mathbf{I}, \label{ma}%
\end{equation}
where $\mathbf{I}$ is a $p$-adic integer (addition shape invariant), and the
nonderived $m$-ary addition $\mathbf{\nu}_{m}$ is the repeated binary sum of
$m$ representatives $\mathbf{r}_{\mathbf{k}}=\mathbf{r}_{\mathbf{k}}\left(
\mathbf{a},\mathbf{b}\right)  $%
\begin{equation}
\mathbf{\nu}_{m}\left[  \mathbf{r}_{\mathbf{k}_{1}},\mathbf{r}_{\mathbf{k}%
_{2}},\ldots,\mathbf{r}_{\mathbf{k}_{m}}\right]  =\mathbf{r}_{\mathbf{k}_{1}%
}+_{p}\mathbf{r}_{\mathbf{k}_{2}}+_{p}\ldots+_{p}\mathbf{r}_{\mathbf{k}_{m}}.
\label{nm}%
\end{equation}

\end{lemma}

\begin{proof}
The condition of closure for the $m$-ary addition $\mathbf{\nu}_{m}$ is
$\mathbf{r}_{\mathbf{k}_{1}}+_{p}\mathbf{r}_{\mathbf{k}_{2}}+_{p}\ldots
+_{p}\mathbf{r}_{\mathbf{k}_{m}}=\mathbf{r}_{\mathbf{k}_{0}}$ in the notation
of (\ref{ab1}). Using (\ref{rk}), it provides $m\mathbf{a}+\mathbf{b\bullet}%
_{p}\left(  \mathbf{k}_{1}+_{p}\mathbf{k}_{2}+_{p}\ldots+_{p}\mathbf{k}%
_{m}\right)  =\mathbf{a}+_{p}\mathbf{b}\bullet_{p}\mathbf{k}_{0}$, which is
equivalent to (\ref{ma}), where $\mathbf{I}=\mathbf{k}_{0}-_{p}\left(
\mathbf{k}_{1}+_{p}\mathbf{k}_{2}+_{p}\ldots+_{p}\mathbf{k}_{m}\right)  $. The
querelement $\mathbf{r}_{\mathbf{\bar{k}}}$ \cite{dor3} satisfies%
\begin{equation}
\mathbf{\nu}_{m}\left[  \mathbf{r}_{\mathbf{k}},\mathbf{r}_{\mathbf{k}}%
,\ldots,\mathbf{r}_{\mathbf{k}},\mathbf{r}_{\mathbf{\bar{k}}}\right]
=\mathbf{r}_{\mathbf{k}},
\end{equation}
which has a unique solution $\mathbf{\bar{k}}=\left(  2-m\right)
\mathbf{k}-\mathbf{I}$. Therefore, each element of $\left[  \mathbf{a}\right]
_{\mathbf{b}}$ is invertible with respect to $\mathbf{\nu}_{m}$, and
$\left\langle \left[  \mathbf{a}\right]  _{\mathbf{b}}\mid\mathbf{\nu}%
_{m}\right\rangle $ is a commutative $m$-ary group.
\end{proof}

\begin{lemma}
The $p$-adic analog of the residue class (\ref{ab1}) is a commutative $n$-ary
semigroup $\left\langle \left[  \mathbf{a}\right]  _{\mathbf{b}}%
\mid\mathbf{\mu}_{n}\right\rangle $, if%
\begin{equation}
\mathbf{a}^{n}-\mathbf{a}=\mathbf{b\mathbf{\bullet}_{p}J}, \label{na}%
\end{equation}
where $\mathbf{J}$ is a $p$-adic integer (multiplication shape invariant), and
the nonderived $m$-ary multiplication $\mathbf{\nu}_{m}$ is the repeated
binary product of $n$ representatives%
\begin{equation}
\mathbf{\mu}_{n}\left[  \mathbf{r}_{\mathbf{k}_{1}},\mathbf{r}_{\mathbf{k}%
_{2}},\ldots,\mathbf{r}_{\mathbf{k}_{n}}\right]  =\mathbf{r}_{\mathbf{k}_{1}%
}\mathbf{r}_{\mathbf{k}_{2}}\ldots\mathbf{r}_{\mathbf{k}_{n}}. \label{mn}%
\end{equation}

\end{lemma}

\begin{proof}
The condition of closure for the $n$-ary multiplication $\mathbf{\mu}_{n}$ is
$\mathbf{r}_{\mathbf{k}_{1}}\bullet_{p}\mathbf{r}_{\mathbf{k}_{2}}\bullet
_{p}\ldots\bullet_{p}\mathbf{r}_{\mathbf{k}_{m}}=\mathbf{r}_{\mathbf{k}_{0}}$.
Using (\ref{rk}) and opening brackets we obtain $n\mathbf{a}+\mathbf{b\bullet
}_{p}\mathbf{J}_{1}=\mathbf{a}+_{p}\mathbf{b}\bullet_{p}\mathbf{k}_{0}$, where
$\mathbf{J}_{1}$ is some $p$-adic integer, which gives (\ref{na}) with
$\mathbf{J}=\mathbf{k}_{0}-_{p}\mathbf{J}_{1}$.
\end{proof}

Combining the conditions (\ref{ma}) and (\ref{na}), we arrive at

\begin{theorem}
The $p$-adic analog of the residue class (\ref{ab1}) becomes a $\left(
m,n\right)  $-ring with $m$-ary addition (\ref{nm}) and $n$-ary multiplication
(\ref{mn})%
\begin{equation}
\mathbb{Z}_{\left(  m,n\right)  }\left(  \mathbf{a}_{q},\mathbf{b}_{q}\right)
=\left\langle \left[  \mathbf{a}_{q}\right]  _{\mathbf{b}_{q}}\mid\mathbf{\nu
}_{m},\mathbf{\mu}_{n}\right\rangle , \label{zab}%
\end{equation}
when the $p$-adic integers $\mathbf{a}_{q},\mathbf{b}_{q}\in\mathbb{Z}_{p}$
are solutions of the equations%
\begin{align}
m\mathbf{a}_{q}  &  =\mathbf{a}_{q}\left(  \operatorname{Mod}_{p}%
\mathbf{b}_{q}\right)  ,\label{aq1}\\
\mathbf{a}_{q}^{n}  &  =\mathbf{a}_{q}\left(  \operatorname{Mod}_{p}%
\mathbf{b}_{q}\right)  . \label{aq2}%
\end{align}

\end{theorem}

\begin{proof}
The conditions (\ref{aq1})--(\ref{aq2}) are equivalent to (\ref{ma}) and
(\ref{na}), respectively, which shows that $\left[  \mathbf{a}_{q}\right]
_{\mathbf{b}_{q}}$ (considered as a set of representatives (\ref{rk})) is
simultaneously an $m$-ary group with respect to $\mathbf{\nu}_{m}$, and an
$n$-ary semigroup with respect to $\mathbf{\mu}_{n}$, and is therefore a
$\left(  m,n\right)  $-ring.
\end{proof}

If we work in the partial sum representation (\ref{y}), the procedure of
finding the digits of $p$-adic integers $\mathbf{a}_{q},\mathbf{b}_{q}%
\in\mathbb{Z}_{p}$ such that $\left[  \mathbf{a}_{q}\right]  _{\mathbf{b}_{q}%
}$ becomes a $\left(  m,n\right)  $-ring with initially fixed arities is
recursive. To find the first digits $\alpha_{0}$ and $\beta_{0}$ that are
ordinary integers, we use the equations (\ref{an1})--(\ref{am2}), and for
their arity shape \textsc{Table \ref{T1}}. Next we consider the second
components of (\ref{y}) to find the digits $\alpha_{1}$ and $\beta_{1}$ of
$\mathbf{a}_{q}$ and $\mathbf{b}_{q}$ by solving the equations (\ref{ma}) and
(\ref{na}) (these having initially given arities $m$ and $n$ from the first
step) by application of the exact formulas from \cite{xu/dai}. In this way, we
can find as many digits $\left(  \alpha_{0},,\alpha_{i_{\max}}\right)  $ and
$\left(  \beta_{0},,\beta_{i_{\max}}\right)  $ of $\mathbf{a}_{q}$ and
$\mathbf{b}_{q}$, as needed for our accuracy preferences in building the
polyadic ring of $p$-adic integers $\mathbb{Z}_{\left(  m,n\right)  }\left(
\mathbf{a}_{q},\mathbf{b}_{q}\right)  $ (\ref{zab}).

Further development and examples will appear elsewhere.

\section{\textsc{Conclusions}}

The study of \textquotedblleft external\textquotedblright\ residue class
properties is a foundational subject in standard number theory. We have
investigated their \textquotedblleft internal\textquotedblright\ properties to
understand the algebraic structure of the representative set of a fixed
residue class. We found that, if the parameters of a class satisfy some
special \textquotedblleft quantization\textquotedblright\ conditions, the set
of representatives becomes a polyadic ring. We introduced the arity shape, a
surjective-like mapping of the residue class parameters to the arity of
addition $m$ and arity of multiplication $n$, which result in commutative
$\left(  m,n\right)  $-rings (see \textsc{Table \ref{T1}}).

We then generalized the approach thus introduced to $p$-adic integers by
defining an analog of a residue class for them. Using the coherent
representation for $p$-adic integers as partial sums we defined 
the $p$-adic analog of the \textquotedblleft quantization\textquotedblright%
\ conditions in a componentwise manner for~when the set of $p$-adic representatives form a polyadic
ring. Finally, we proposed a recursive procedure to find any desired digits of
the $p$-adic residue class parameters.

The presented polyadic algebraic structure of $p$-adic numbers may lead to new
symmetries and features in $p$-adic mathematical physics and the corresponding
particle models.

\bigskip

\textbf{Acknowledgements.} The author is grateful to Branko Dragovich, Mike
Hewitt, Vladimir Tkach and Raimund Vogl for useful discussions and valuable help.

\pagestyle{emptyf}
\mbox{}

\end{document}